\documentclass[11pt]{article}
\usepackage{amsmath,amssymb,amsthm}
\usepackage{amscd}
\newtheorem{theorem}{Theorem}[section]
\newtheorem{lemma}[theorem]{Lemma}
\newtheorem{proposition}[theorem]{Proposition}

\theoremstyle{plain}

\theoremstyle{definition}
\newtheorem{definition}[theorem]{Definition}

\numberwithin{equation}{section}

\renewcommand{\labelenumi}{\textup{(\theenumi)}}

\newcommand{\Homeo}{\operatorname{Homeo}}

\newcommand{\id}{\operatorname{id}}
\newcommand{\Ker}{\operatorname{Ker}}
\newcommand{\Min}{\operatorname{Min}}

\newcommand{\Ad}{\operatorname{Ad}}

\def\det{{{\operatorname{det}}}}

\newcommand{\N}{\mathbb{N}}

\newcommand{\R}{\mathbb{R}}
\newcommand{\T}{\mathbb{T}}
\newcommand{\Z}{\mathbb{Z}}
\newcommand{\Zp}{{\mathbb{Z}}_+}

\title{Uniformly continuous orbit equivalence of Markov shifts and gauge actions on Cuntz--Krieger algebras}
\author{Kengo Matsumoto \\
Department of Mathematics \\
Joetsu University of Education \\
Joetsu, 943-8512, Japan
}
\date{}

\begin{document}
\maketitle

\def\det{{{\operatorname{det}}}}

\begin{abstract}
We introduce a notion of uniformly continuous orbit equivalence 
as a subequivalence relation of continuous orbit equivalence of 
 one-sided topological Markov shifts.
It is described in terms of gauge actions on the associated Cuntz--Krieger algebras and continuous full groups of the Markov shifts.
\end{abstract}




\def\OA{{{\mathcal{O}}_A}}
\def\OB{{{\mathcal{O}}_B}}
\def\OZ{{{\mathcal{O}}_Z}}
\def\OTA{{{\mathcal{O}}_{\tilde{A}}}}
\def\SOA{{{\mathcal{O}}_A}\otimes{\mathcal{K}}}
\def\SOB{{{\mathcal{O}}_B}\otimes{\mathcal{K}}}
\def\SOZ{{{\mathcal{O}}_Z}\otimes{\mathcal{K}}}
\def\SOTA{{{\mathcal{O}}_{\tilde{A}}\otimes{\mathcal{K}}}}
\def\DA{{{\mathcal{D}}_A}}
\def\DB{{{\mathcal{D}}_B}}
\def\DZ{{{\mathcal{D}}_Z}}
\def\FA{{{\mathcal{F}}_A}}
\def\FB{{{\mathcal{F}}_B}}
\def\DTA{{{\mathcal{D}}_{\tilde{A}}}}
\def\SDA{{{\mathcal{D}}_A}\otimes{\mathcal{C}}}
\def\SDB{{{\mathcal{D}}_B}\otimes{\mathcal{C}}}
\def\SDZ{{{\mathcal{D}}_Z}\otimes{\mathcal{C}}}
\def\SDTA{{{\mathcal{D}}_{\tilde{A}}\otimes{\mathcal{C}}}}
\def\Max{{{\operatorname{Max}}}}
\def\AF{{{\operatorname{AF}}}}
\def\Per{{{\operatorname{Per}}}}
\def\PerB{{{\operatorname{PerB}}}}
\def\Homeo{{{\operatorname{Homeo}}}}
\def\HA{{{\frak H}_A}}
\def\HB{{{\frak H}_B}}
\def\HSA{{H_{\sigma_A}(X_A)}}
\def\Out{{{\operatorname{Out}}}}
\def\Aut{{{\operatorname{Aut}}}}
\def\Ad{{{\operatorname{Ad}}}}
\def\Inn{{{\operatorname{Inn}}}}
\def\det{{{\operatorname{det}}}}
\def\exp{{{\operatorname{exp}}}}
\def\cobdy{{{\operatorname{cobdy}}}}
\def\Ker{{{\operatorname{Ker}}}}
\def\ind{{{\operatorname{ind}}}}
\def\id{{{\operatorname{id}}}}
\def\supp{{{\operatorname{supp}}}}
\def\co{{{\operatorname{co}}}}
\def\Sco{{{\operatorname{Sco}}}}
\def\ActA{{{\operatorname{Act}_{\DA}(\mathbb{T},\OA)}}}
\def\ActB{{{\operatorname{Act}_{\DB}(\mathbb{T},\OB)}}}
\def\RepOA{{{\operatorname{Rep}(\mathbb{T},\OA)}}}
\def\RepDA{{{\operatorname{Rep}(\mathbb{T},\DA)}}}
\def\RepDB{{{\operatorname{Rep}(\mathbb{T},\DB)}}}
\def\U{{{\mathcal{U}}}}
\def\coe{{{\operatorname{coe}}}}
\def\scoe{{{\operatorname{scoe}}}}
\def\uoe{{{\operatorname{uoe}}}}
\def\ucoe{{{\operatorname{ucoe}}}}
\def\event{{{\operatorname{event}}}}

\section{Introduction and Preliminaries}

For an $N \times N$ irreducible non permutation matrix $A= [A(i,j)]_{i,j=1}^N$
with entries in $\{0,1\}$,
the shift space $X_A$  of one-sided topological Markov shift $(X_A,\sigma_A)$
is defined by
\begin{equation}
X_A = \{ (x_n)_{n \in \N} \in \{ 1,\dots,N \}^\N
 \mid A(x_n,x_{n+1}) =1 \text{ for all } n \in \N \}
\end{equation}
where $\N$ denotes the set of positive integers.
It is a compact Hausdorff space by the relative topology of 
$\{1,\dots, N \}^\N $ with the infinite product topology.
It has a shift transformation $\sigma_A$ defined by 
$\sigma_A((x_n)_{n \in \N}) = (x_{n+1})_{n \in \N}$.
The topological dynamical system 
$(X_A,\sigma_A)$ is called the one-sided topological Markov shift defined by the matrix 
$A$.
The author has introduced a notion of continuous orbit equivalence in the class of one-sided topological Markov shifts to classify the Cuntz--Krieger algebras in \cite{MaPacific}.
One-sided topological Markov shifts
$(X_A, \sigma_A)$ and $(X_B,\sigma_B)$ 
are said to be continuously orbit equivalent written
$(X_A, \sigma_A)\underset{\coe}{\sim}(X_B,\sigma_B)$ 
if there exist a homeomorphism
$h: X_A \rightarrow X_B$ 
and continuous functions 
$k_1,l_1: X_A\rightarrow \Zp,
 k_2,l_2: X_B\rightarrow \Zp 
$
such that
\begin{align}
\sigma_B^{k_1(x)} (h(\sigma_A(x))) 
& = \sigma_B^{l_1(x)}(h(x))
\quad \text{ for} \quad 
x \in X_A,  \label{eq:orbiteq1x} \\
\sigma_A^{k_2(y)} (h^{-1}(\sigma_B(y))) 
& = \sigma_A^{l_2(y)}(h^{-1}(y))
\quad \text{ for } \quad 
y \in X_B \label{eq:orbiteq2y}
\end{align}
where $\Zp$ denotes the set of nonnegative integers.
The functions $c_1 = l_1 - k_1, c_2 = l_2 - k_2$
are called the cocycle functions for $h, h^{-1}$, respectively. 
The continuous orbit equivalence class of $(X_A,\sigma_A)$
naturally yields a subgroup of homeomorphism group on $X_A$ which is called 
the continuous full group written $\Gamma_A$. 
The group
$\Gamma_A$ consists of 
homeomorphisms $\tau$ on $X_A$ such that there exist continuous functions 
$k, l : X_A \rightarrow \Zp$ 
such that 
\begin{equation}
\sigma_A^{k(x)}(\tau(x) )=\sigma_A^{l(x)}(x)
\quad\text{ 
for all } x \in X_A. \label{eq:tau}
\end{equation}
The group $\Gamma_A$ has been written $[\sigma_A]_c$ in the earlier papers
(\cite{MaPacific}, \cite{MaDCDS}).
The Cuntz--Krieger algebra $\OA$ is defined by the universal $C^*$-algebra
generated by partial isometries 
$S_1,\dots,S_N$ satisfying the relations:
\begin{equation}
\sum_{j=1}^N S_j S_j^* =1,\qquad
S_i^* S_i = \sum_{j=1}^N A(i,j) S_j S_j^*, \qquad i=1,\dots,N. \label{eq:CK}
\end{equation}
We denote by 
$\DA$ 
the $C^*$-subalgebra of $\OA$
generated by the projections of the form:
$S_{\mu_1}\cdots S_{\mu_n}S_{\mu_n}^* \cdots S_{\mu_1}^*,
\mu_1,\dots, \mu_n =1,\dots,N$.
The subalgebra $\DA$ is naturally isomorphic to the commutative 
$C^*$-algebra $C(X_A)$
of the complex valued continuous functions on $X_A$
by identifying the projection
$S_{\mu_1}\cdots S_{\mu_n}S_{\mu_n}^* \cdots S_{\mu_1}^*
$
with the characteristic function
$\chi_{U_{\mu_1\cdots \mu_n}} \in C(X_A)$
of the cylinder set
$U_{\mu_1\cdots \mu_n}$
for the word
${\mu_1\cdots \mu_n}$.
 
H. Matui and the author have finally reached the following classification result:
\begin{theorem}[\cite{MMKyoto}, cf. \cite{MaPacific}, \cite{MaPAMS2013}, \cite{MaIsrael2015}, \cite{MatuiCrelle}] \label{thm:mm}
Let $A$ and $B$ be irreducible, non permutation matrices with entries in $\{0,1\}$.
The following are equivalent.
\begin{enumerate}
\renewcommand{\theenumi}{\roman{enumi}}
\renewcommand{\labelenumi}{\textup{(\theenumi)}}
\item 
$(X_A, \sigma_A)$ and $(X_B,\sigma_B)$
are continuously orbit equivalent.
\item The Cuntz--Krieger algebras 
$\OA$ and $\OB$ are isomorphic and $\det(\id - A) = \det(\id - B)$.   
\item There exists an isomorphism
$\Phi:\OA \rightarrow \OB$ 
such that $\Phi(\DA) = \DB$.
\item The continuous full groups $\Gamma_A$ and $\Gamma_B$
are isomorphic as groups.\item
There exists a homeomorphism $h: X_A \longrightarrow X_B$ such that 
$
h\circ \Gamma_A \circ h^{-1} = \Gamma_B.
$
\end{enumerate}
\end{theorem}

For $t \in {\mathbb{R}}/\Z = {\mathbb{T}}$,
the correspondence
$S_i \rightarrow e^{2 \pi\sqrt{-1}t}S_i,
\, i=1,\dots,N$
gives rise to an automorphism
of $\OA$ denoted by $\rho^A_t$.
The automorphisms $\rho^A_t, t \in {\mathbb{T}}$
yield an action of ${\mathbb{T}}$
on $\OA$ called the gauge action.
The gauge action is a basic tool to analyze the structure of the Cuntz--Krieger algebra
$\OA$ as in \cite{CK} and is closely related to dynamical structure of the underlying topological Markov shift $(X_A,\sigma_A)$. 
Let us denote by $\FA$ the fixed point subalgebra of  $\OA$ under the gauge action
$\rho^A$. 
It is well-known that  $\FA$ is an AF algebra whose $K_0$-group is known as the dimension group of the underlying topological Markov shift.
The subalgebra of $\FA$ consisting of diagonal elements coincides with the maximal commutative $C^*$-subalgebra $\DA$ defined above.
 There is a discrete subgroup of $\Gamma_A$ 
which gives rise to unitaries of finite dimensional subalgebras of the AF algebra $\FA$. 
It is a group of homeomorphism $\tau$ in $\Gamma_A$ 
for which one may take $k(x) = l(x)$ in \eqref{eq:tau}.
Since the function $k(x)$ is continuous, it may be chosen to be a constant number written
$K_{\tau}$.
We write the group
as $\Gamma_A^{\AF}$, 
which has been written ${[\sigma_A]}_{\AF}$ in \cite[Section 7]{MaPacific}.
One-sided topological Markov shifts
$(X_A, \sigma_A)$ and $(X_B,\sigma_B)$ 
are said to be {\it uniformly orbit equivalent}\/ (\cite[Section 7]{MaPacific}) 
if there exist a homeomorphism
$h: X_A \rightarrow X_B$ 
such that for any $\tau_1 \in \Gamma_A^{\AF}, \tau_2 \in \Gamma_B^{\AF}, $
there exist continuous functions 
$k_1: X_A\rightarrow \Zp,
 k_2: X_B\rightarrow \Zp 
$
such that
\begin{align}
\sigma_B^{k_1(x)} (h(\tau_1(x))) 
& = \sigma_B^{k_1(x)}(h(x))
\quad \text{ for} \quad 
x \in X_A,  \label{eq:uorbiteq1x} \\
\sigma_A^{k_2(y)} (h^{-1}(\tau_2(y))) 
& = \sigma_A^{k_2(y)}(h^{-1}(y))
\quad \text{ for } \quad 
y \in X_B. \label{eq:uorbiteq2y}
\end{align}
This situation is written
$(X_A, \sigma_A)\underset{\uoe}{\sim}(X_B,\sigma_B).$ 
In this case, both functions $k_1, k_2$ are continuous, so that 
one may take them to be natural numbers, written $K_{\tau_1}, K_{\tau_2}$, such that  
\begin{align}
\sigma_B^{K_{\tau_1}} (h(\tau_1(x))) 
& = \sigma_B^{K_{\tau_1}}(h(x))
\quad \text{ for} \quad 
x \in X_A,  \label{eq:uorbiteq1K} \\
\sigma_A^{K_{\tau_2}} (h^{-1}(\tau_2(y))) 
& = \sigma_A^{K_{\tau_2}}(h^{-1}(y))
\quad \text{ for } \quad 
y \in X_B. \label{eq:uorbiteq2K}
\end{align}
The following proposition has been seen in \cite{MaPacific}.
\begin{proposition}[cf. {\cite[Theorem 7.4]{MaPacific}}, \cite{Kr}, \cite{Kr2}, \cite{SV}]\label{prop:1.3}
Let $A$ and $B$ be irreducible, non permutation matrices with entries in $\{0,1\}$.
The following are equivalent.
\begin{enumerate}
\renewcommand{\theenumi}{\roman{enumi}}
\renewcommand{\labelenumi}{\textup{(\theenumi)}}
\item 
$(X_A, \sigma_A)$ and $(X_B,\sigma_B)$
are uniformly orbit equivalent.
\item There exists an isomorphism
$\Phi:\FA \rightarrow \FB$ 
such that $\Phi(\DA) = \DB$.
\item
There exists a homeomorphism $h: X_A \longrightarrow X_B$ such that 
$
h\circ \Gamma_A^{\AF} \circ h^{-1} = \Gamma_B^{\AF}.
$
\end{enumerate}
\end{proposition}

We note that canonical maximal abelian $C^*$-subalgebras of $\FA$ are unique
up to isomorphism on $\FA$ (cf. \cite{SV}). 
Hence the above condition (ii) is equivalent to the condition that $\FA$ is isomorphic to $\FB$.

In this paper, we will study relationships among classification of gauge actions on Cuntz--Krieger algebras, continuous orbit equivalence and continuous full groups.
We will show an analogue of Theorem \ref{thm:mm} for gauge actions
referring to Proposition \ref{prop:1.3}.
 In the definition of continuous orbit equivalence, if one may take 
$l_1(x) = k_1(x) +1, x \in X_A$
in \eqref{eq:orbiteq1x}
and
$l_2(y) = k_2(y) +1, y \in X_B$
in \eqref{eq:orbiteq2y},
then $(X_A,\sigma_A)$ and $(X_B,\sigma_B)$
are said to be {\it eventually one-sided conjugate\/} (\cite{MaPre2015}).
This situation is written
$(X_A, \sigma_A)\underset{\event}{\approx}(X_B,\sigma_B).$ 
In this case,
one may take the functions  
$k_1, k_2$ to be the constant functions 
taking its values 
$
K_1=\Max \{k_1(x) \mid x \in X_A\},
K_2=\Max \{k_2(y) \mid y \in X_B\}$.
Hence it is easy to see that
$(X_A, \sigma_A)\underset{\event}{\approx}(X_B,\sigma_B)$ 
if and only if
 there exist a homeomorphism
$h: X_A \rightarrow X_B$ 
and  natural numbers $K_1, K_2 \in \N$
such that 
\begin{align}
\sigma_B^{K_1} (h(\sigma_A(x))) 
& = \sigma_B^{K_1+1}(h(x))
\quad \text{ for} \quad 
x \in X_A,  \label{eq:oneorbiteq1K} \\
\sigma_A^{K_2} (h^{-1}(\sigma_B(y))) 
& = \sigma_A^{K_2+1}(h^{-1}(y))
\quad \text{ for } \quad 
y \in X_B. \label{eq:oneorbiteq2K}
\end{align}

For the eventually conjugacy, we have obtained the following result.
\begin{proposition}[{\cite[Corollary 3.5]{MaPre2015}}]\label{prop:1.4}
Let $A$ and $B$ be irreducible, non permutation matrices with entries in $\{0,1\}$.
$(X_A, \sigma_A)$ and $(X_B,\sigma_B)$
are eventually one-sided conjugate
if and only if there exists an isomorphism
$\Phi:\OA \rightarrow \OB$ 
such that 
\begin{equation*}
\Phi(\DA) = \DB
\quad \text{  and } \quad
\Phi \circ \rho^A_t = \rho^B_t \circ \Phi, 
\qquad t \in {\mathbb{T}}.
\end{equation*}
\end{proposition}

We introduce a unified notion of continuous orbit equivalence and uniformly
orbit equivalence in the following way. 
\begin{definition}
One-sided topological Markov shifts
$(X_A, \sigma_A)$ and $(X_B,\sigma_B)$ 
are said to be {\it uniformly continuously orbit equivalent}\/ 
if there exist a homeomorphism
$h: X_A \rightarrow X_B$ 
and continuous functions 
$k_1,l_1: X_A\rightarrow \Zp,
 k_2,l_2: X_B\rightarrow \Zp 
$
satisfying \eqref{eq:orbiteq1x} and \eqref{eq:orbiteq2y}
and
for any $\tau_1 \in \Gamma_A^{\AF}, \tau_2 \in \Gamma_B^{\AF}, $
there exist natural numbers $K_{\tau_1}, K_{\tau_2} \in \N$ 
satisfying \eqref{eq:uorbiteq1K}, \eqref{eq:uorbiteq2K}. 
\end{definition}
This situation is written
$(X_A, \sigma_A)\underset{\ucoe}{\sim}(X_B,\sigma_B).$

In the presented paper, we will show the following theorem.
\begin{theorem}\label{thm:main}
Let $A$ and $B$ be irreducible, non permutation matrices with entries in $\{0,1\}$.
The following are equivalent.
\begin{enumerate}
\renewcommand{\theenumi}{\roman{enumi}}
\renewcommand{\labelenumi}{\textup{(\theenumi)}}
\item 
$(X_A, \sigma_A)$ and $(X_B,\sigma_B)$
are eventually one-sided conjugate.
\item 
$(X_A, \sigma_A)$ and $(X_B,\sigma_B)$
are uniformly continuously orbit equivalent.
\item There exists an isomorphism
$\Phi:\OA \rightarrow \OB$ 
such that 
\begin{equation*}
\Phi(\DA) = \DB
\quad \text{  and } \quad
\Phi \circ \rho^A_t = \rho^B_t \circ \Phi, 
\qquad t \in {\mathbb{T}}
\end{equation*}
\item There exists an isomorphism
$\Phi:\OA \rightarrow \OB$ 
such that 
\begin{equation*}
\Phi(\DA) = \DB
\quad \text{  and } \quad
\Phi(\FA) = \FB
\end{equation*}
\item There exists an isomorphism $\xi: \Gamma_A \longrightarrow \Gamma_B$
of groups such that  $\xi(\Gamma_A^{\AF}) = \Gamma_B^{\AF}.$ 
\item
There exists a homeomorphism $h: X_A \longrightarrow X_B$ such that 
\begin{equation*}
h\circ \Gamma_A \circ h^{-1} = \Gamma_B \quad \text{ and }
\quad
 h\circ \Gamma_A^{\AF} \circ h^{-1} = \Gamma_B^{\AF}.
\end{equation*}
\end{enumerate}
\end{theorem}

Before ending this section,
we provide several notations and  a lemma which will be useful in the proof of 
Theorem \ref{thm:main} in the next section.
For $n\in \N$, we denote by $B_n(X_A)$ the set of admissible words of $X_A$
with length $n$.
 We denote by $C(X_A,\Z)$
the set of integer valued continuous functions on $X_A$. 
It has a natural structure of abelian group by pointwise addition of functions.
For $f \in C(X_A,\Z)$ and $k \in \N$,
we set 
$f^k(x) = \sum_{i=0}^{k-1}f(\sigma_A^i(x))$ for $x \in X_A$. 
As $f$ is regarded as an element of $\DA$, 
we may define an automorphism $\rho^{A,f}_t$ for $t \in \R/\Z$ on $\OA$ 
by setting
\begin{equation*}
\rho^{A,f}_t(S_i) = e^{2 \pi \sqrt{-1} t f}S_i, \qquad i=1,\dots,N 
\end{equation*}
which gives rise to an action of $\T$ on $\OA$.
If in particular $f\equiv 1_{X_A}$, the action
$\rho^{A,f}$ is the gauge action $\rho^{A}.$

Suppose that $(X_A, \sigma_A)\underset{\coe}{\sim}(X_B,\sigma_B).$
Let $h:X_A\rightarrow X_B$ be a homeomorphism
and
$k_1,l_1: X_A\rightarrow \Zp,
 k_2,l_2: X_B\rightarrow \Zp 
$
be continuous functions satisfying \eqref{eq:orbiteq1x} and \eqref{eq:orbiteq2y}
respectively.
As in \cite{MMETDS},
the homomorphism
$\Psi_h: C(X_B,\Z) \rightarrow C(X_A,\Z)$
defined by
\begin{equation}
\Psi_{h}(g)(x)
= \sum_{i=0}^{l_1(x)-1} g(\sigma_B^i(h(x))) 
- \sum_{j=0}^{k_1(x)-1} g(\sigma_B^j(h(\sigma_A(x))))
\label{eq:Psih}
\end{equation}
for
$g \in C(X_B,\Z), \ x \in X_A$
and its inverse 
$\Psi_{h^{-1}}: C(X_A,\Z) \rightarrow C(X_B,\Z)$
gives rise to an isomorphism of  groups.
By \cite{MaPre2015}, there exists an isomorphism
$\Phi: \OA \longrightarrow \OB$ such that 
\begin{equation*}
\Phi(\DA) = \DB
\quad \text{  and } \quad
\Phi \circ \rho^{A,\Psi_h(g)}_t = \rho^{B,g}_t \circ \Phi, 
\qquad g \in C(X_B,\Z), \, t \in {\mathbb{T}}.
\end{equation*}

\begin{lemma}\label{lem:useful}
Let
$(X_A, \sigma_A)$ and $(X_B,\sigma_B)$ 
be continuously  orbit equivalent given by 
a homeomorphism  
$h: X_A \rightarrow X_B$ 
 with continuous functions 
$k_1,l_1: X_A\rightarrow \Zp,
 k_2,l_2: X_B\rightarrow \Zp 
$
satisfying
\eqref{eq:orbiteq1x}
and
\eqref{eq:orbiteq2y}.
Assume that either of the  cocycle functions 
$c_1 = l_1 - k_1$ on $X_A$ or 
$c_2 = l_2 - k_2$ on $X_B$
is  constant.
Then both of the functions are $1$.
\end{lemma}
\begin{proof}
Suppose that $c_1$ is a constant function taking value $C_1$.
By \cite[Lemma 3.3]{MMKyoto},
the equality 
\begin{equation}
k_1^{l_2(y)}(h^{-1}(y)) + l_1^{k_2(y)}(h^{-1}(\sigma_B(y))) +1 
= k_1^{k_2(y)}(h^{-1}(\sigma_B(y))) + l_1^{l_2(y)}(h^{-1}(y)) \label{eq:13}
\end{equation}
holds. Since
\begin{align*}
  & l_1^{l_2(y)}(h^{-1}(y)) - k_1^{l_2(y)}(h^{-1}(y)) \\
= & \sum_{i=0}^{l_2(y)-1} l_1(\sigma_A^i(h^{-1}(y))) -
   \sum_{i=0}^{l_2(y)-1} k_1(\sigma_A^i(h^{-1}(y))) 
= l_2(y) C_1,
\end{align*}
and similarly
\begin{equation*}
 l_1^{k_2(y)}(h^{-1}(\sigma_B(y))) - k_1^{k_2(y)}(h^{-1}(\sigma_B(y)))
= k_2(y) C_1,
\end{equation*}
the equality \eqref{eq:13} ensures us that
$l_2(y)C_1 - k_2(y) C_1 =1$
so that 
$c_2(y) C_1 =1$ for all $y \in X_B$. 
Hence the function $c_2$ is also constant whose value is written $C_2$. 
We then have $C_1\cdot C_2 = 1$.
Since both $C_1$ and $C_2$ are integers,
we have $C_1 = C_2 =1$ or  $C_1 = C_2 =-1.$
As in \cite[Corollary 5.9]{MMETDS},
we see that
$$
\sum_{i=0}^{r-s-1}c_1(\sigma_A^{s+i}(x)) >0
\quad
\text{  for } x \in X_A
\text{  with } 
\sigma_A^r(x) =\sigma_A^s(x), r-s>0.
$$
Hence $C_1$ must be positive, so that we have
$C_1 = C_2 = 1$.
\end{proof}

\section{Proof of Theorem \ref{thm:main}}
This section is devoting to proving Theorem \ref{thm:main}.

{ (i) $\Longrightarrow$ (ii):\,\,}
Suppose that 
$(X_A, \sigma_A)\underset{\event}{\approx}(X_B,\sigma_B).$
Take a homeomorphism $h:X_A\longrightarrow X_B$ and $K_1, K_2 \in \N$
satisfying \eqref{eq:oneorbiteq1K}, \eqref{eq:oneorbiteq2K}.
For any 
$\tau_1\in \Gamma_A^{\AF},
$
there exists
$K_{\tau_1}$
such that 
$$
\sigma_A^{K_{\tau_1}}(\tau_1(x))  = \sigma_A^{K_{\tau_1}}(x), \qquad x \in X_A.
$$
By \eqref{eq:oneorbiteq1K}, 
we have
$$
\sigma_B^{K_1} (h(\sigma_A^{K_{\tau_1}}(x))) 
 = \sigma_B^{K_1+K_{\tau_1}}(h(x))
\quad \text{ for} \quad 
x \in X_A, 
$$ 
so that 
$$
\sigma_B^{K_1+K_{\tau_1}} (h(\tau_1(x))) 
 =\sigma_B^{K_1} (h(\sigma_A^{K_{\tau_1}}(\tau_1(x)))) 
  = \sigma_B^{K_1+K_{\tau_1}}(h(x))
\quad \text{ for} \quad 
x \in X_A.
$$ 
Similarly there exists $K_{\tau_2} \in \N$ for 
$\tau_2\in \Gamma_B^{\AF}$ such that 
$$
\sigma_A^{K_2+K_{\tau_2}} (h^{-1}(\tau_2(y))) 
 = \sigma_A^{K_2+K_{\tau_2}}(h^{-1}(y))
\quad \text{ for } \quad 
y \in X_B,
$$
so that
$(X_A, \sigma_A)\underset{\ucoe}{\sim}(X_B,\sigma_B).$



{ (ii) $\Longrightarrow$ (i):\,\,}
A point $x \in X_A$ is said to be eventually periodic
if there exist
$r, s \in \Zp$ with $r -s >0$ such that 
 $\sigma_A^r(x) = \sigma_A^s(x)$.
As the matrix $A$ is irreducible and not any permutations, 
the set 
$X_A^{nep}$ of non eventually periodic points of $X_A$ is dense in $X_A$. 
Suppose that 
$(X_A, \sigma_A)\underset{\ucoe}{\sim}(X_B,\sigma_B).$
Take a homeomorphism
$h: X_A \rightarrow X_B$ 
and continuous functions 
$k_1,l_1: X_A\rightarrow \Zp,
 k_2,l_2: X_B\rightarrow \Zp 
$
satisfying the conditions
\eqref{eq:orbiteq1x},
\eqref{eq:orbiteq2y}
and
\eqref{eq:uorbiteq1K},
\eqref{eq:uorbiteq2K}.
Put
$c_i = l_i - k_i, i=1,2$.
We will first show that both $c_1$ and $c_2$ are constant. 
Suppose that $c_1$ is not constant.
We may find $z \in X_A^{nep}$ and $\tau \in \Gamma_A^{\AF}$
such that 
$c_1(z) \ne c_1(\tau(z))$.
Since we may take 
$k \in \N$ such that 
$\sigma_A^k(z) = \sigma_A^k(\tau(z))$,
the set
$$
S_0 = \{k \in \N \mid \exists x \in X_A^{nep} ; \exists \tau \in \Gamma_A^{\AF};
  c_1(x) \ne c_1(\tau(x)),  \sigma_A^k(x) = \sigma_A^k(\tau(x)) \}
$$
is not empty. We put 
$K_0 = \Min S_0$.
Take $x \in S_0$ and $\tau \in \Gamma_A^{\AF}$
such that 
\begin{equation}
c_1(x) \ne c_1(\tau(x)), \qquad 
\sigma_A^{K_0}(x) = \sigma_A^{K_0}(\tau(x)). \label{eq:11}
\end{equation}
By \eqref{eq:orbiteq1x} or \cite[Lemma 3.1]{MMETDS}, we have
$$
\sigma_B^{k_1^{K_0}(x)}(h(\sigma_A^{K_0}(x))) =\sigma_B^{l_1^{K_0}(x)}(h(x)),
$$
so that
$$
\sigma_B^{k_1^{K_0}(x)}(h(\sigma_A^{K_0}(\tau(x)))) =\sigma_B^{l_1^{K_0}(x)}(h(x)).
$$
Hence we have 
$$
\sigma_B^{k_1^{K_0}(x) + k_1^{K_0}(\tau(x))}(h(\sigma_A^{K_0}(\tau(x)))) 
=\sigma_B^{l_1^{K_0}(x) +k_1^{K_0}(\tau(x))}(h(x)),
$$
and
\begin{equation*}
\sigma_B^{k_1^{K_0}(x) + l_1^{K_0}(\tau(x))}(h(\tau(x))) 
=\sigma_B^{l_1^{K_0}(x) +k_1^{K_0}(\tau(x))}(h(x)). \label{eq:2}
\end{equation*}
Since there exists $K \in \N$ such that 
\begin{equation*}
\sigma_B^K(h(\tau(x))) = \sigma_B^K(h(x)), \label{eq:3} 
\end{equation*}
we have
\begin{equation*}
\sigma_B^{k_1^{K_0}(x) + l_1^{K_0}(\tau(x)) +K}(h(\tau(x))) 
=\sigma_B^{l_1^{K_0}(x) +k_1^{K_0}(\tau(x)) +K}(h(\tau(x))). 
\end{equation*}
By the discussion in \cite[Section 6]{MMETDS}, 
homeomorphism $h$ giving rise to a continuous orbit equivalence
preserves eventually periodic points.
As $\tau(x) \in X_A^{nep}$, we see 
$h(\tau(x)) \in X_A^{nep}$
so that
\begin{equation*}
k_1^{K_0}(x) + l_1^{K_0}(\tau(x)) +K
=l_1^{K_0}(x) +k_1^{K_0}(\tau(x)) +K
\end{equation*}
which implies 
 $l_1^{K_0}(x) - k_1^{K_0}(x) = l_1^{K_0}(\tau(x)) - k_1^{K_0}(\tau(x)) 
$
and hence
$c_1^{K_0}(x) = c_1^{K_0}(\tau(x)).$
This means  that
\begin{equation}
\sum_{i=0}^{K_0-1} c_1(\sigma_A^i(x))  
=
\sum_{i=0}^{K_0-1} c_1(\sigma_A^i(\tau(x))).
\label{eq:31}
\end{equation}
If there exists $m\in \N$ such that 
$1\le m \le K_0 -1$ and
$c_1(\sigma_A^m(x)) \ne 
c_1(\sigma_A^m(\tau(x)))$,
Put
$\bar{x} = \sigma_A^m(x)$.
One may find $\bar{\tau} \in \Gamma_A^{\AF}$
such that 
$\bar{\tau}(\bar{x})=\sigma_A^m(\tau(x)).$
As
$c_1(\sigma_A^m(x)) \ne 
c_1(\sigma_A^m(\tau(x)))$
and
$\sigma_A^{K_0}(x) = \sigma_A^{K_0}(\tau(x)),
$
we have
\begin{equation*}
c_1(\bar{x})) \ne c_1(\bar{\tau}(\bar{x})),
\qquad
\sigma_A^{K_0-m}(\bar{x}) = \sigma_A^{K_0-m}(\bar{\tau}(\bar{x})).
\end{equation*} 
This is a contradiction of the minimality of $K_0$.
Hence we see that 
\begin{equation}
c_1(\sigma_A^m(x)) =c_1(\sigma_A^m(\tau(x))) \quad\text{ for all } m \text{ with }
1\le m \le K_0-1.\label{eq:32}
\end{equation}
By \eqref{eq:31} and \eqref{eq:32},
we see that  $c_1(x) = c_1(\tau(x))$, a contradiction to \eqref{eq:11}.
Hence we conclude that $c_1$ is constant.
By Lemma \ref{lem:useful}, we know that
$c_1 \equiv c_2 \equiv1$,
so that
$(X_A, \sigma_A)\underset{\event}{\approx}(X_B,\sigma_B).$


{ (i) $\Longleftrightarrow$ (iii):\,\,} These implications come from Proposition \ref{prop:1.4}.


{ (iii) $\Longrightarrow$ (iv):\,\,} This implication is obvious.


{ (iv) $\Longrightarrow$ (iii):\,\,} 
Assume that  there exists an isomorphism
$\Phi:\OA \rightarrow \OB$ 
satisfying
$\Phi(\DA) = \DB$
and 
$\Phi(\FA) = \FB.$
Put
$\gamma_t^B = \Phi \circ \rho^A_t \circ \Phi^{-1}$
which is an automorphism for each $t \in \T$.
Since $\Phi(\FA) = \FB$, we see that
$\gamma_t^B(a) =a$ for all $a \in \FB$.
Let us denote by $S_1,\dots,S_M$ the generating partial isometries 
of $\OB$ satisfying the relations \eqref{eq:CK} for the matrix $B$.
Put $W_t =\sum_{i=1}^M\gamma_t^B(S_i) S_i^*$.
For $\mu=(\mu_1,\dots,\mu_n),
        \nu=(\nu_1,\dots,\nu_n)\in B_n(X_B)$, 
we put
$S_\mu = S_{\mu_1}\cdots S_{\mu_n}, 
  S_\nu = S_{\nu_1}\cdots S_{\nu_n}.
$
It then follows that
\begin{align*}
W_t S_\mu S_\nu^* 
=& \sum_{i=1}^M\gamma_t^B(S_i) S_i^* S_{\mu_1}\cdots S_{\mu_n} 
   S_{\nu_n}^* \cdots S_{\nu_1}^* \\ 
=& \gamma_t^B(S_{\mu_1}) S_{\mu_1}^* S_{\mu_1}\cdots S_{\mu_n} 
   S_{\nu_n}^*\cdots S_{\nu_1}^* \\
=& \gamma_t^B(S_{\mu_1} S_{\mu_1}^* S_{\mu_1}\cdot 
   S_{\mu_2}\cdots S_{\mu_n} S_{\nu_n}^* \cdots S_{\nu_1}^* S_{\nu_1}) S_{\nu_1}^* \\
=& \gamma_t^B(S_{\mu}S_{\nu}^* S_{\nu_1}) S_{\nu_1}^* \\
=& S_{\mu}S_{\nu}^* \gamma_t^B(S_{\nu_1}) S_{\nu_1}^* \\
=& S_{\mu}S_{\nu}^*S_{\nu_1} S_{\nu_1}^* \sum_{i=1}^M\gamma_t^B(S_i) S_i^* \\
=& S_{\mu}S_{\nu}^* W_t
\end{align*}
so that $W_t$ commutes with all elements of $\FB.$
Since an element of $\OA$ commuting with $\FA$ must be scalar,
 the correspondence
$t \in \T \longrightarrow W_t \in {\mathbb{C}}$
gives rise to a character of $\T$.
One may find an integer 
$c_B \in \Z$ such that 
$W_t = e^{2 \pi \sqrt{-1}c_B t}, \, t \in \R/\Z = \T$
so that  
$\gamma_t^B(S_i) = e^{2 \pi \sqrt{-1}c_B t} S_i = \rho^{B,c_B}_t(S_i), i=1,\dots,M.$
Hence we have
\begin{equation}
\Phi \circ \rho_t^A = \rho^{B,c_B}_t \circ \Phi , \qquad t \in \T. \label{eq:rhobct}
\end{equation}
Now $\Phi:\OA \longrightarrow \OB$ is an isomorphism such that $\Phi(\DA) = \DB$.
By \cite{MaPacific},
there is a homeomorphism 
$h:X_A \longrightarrow X_B$ which gives rise to continuous orbit equivalence
between $(X_A,\sigma_A)$ and $(X_B, \sigma_B)$ such that
$\Phi(f) = f \circ h^{-1}$ for $f \in C(X_A) =\DA.$ 
For the homeomorphism 
$h:X_A \longrightarrow X_B$,
let $\Phi_h : \OA \longrightarrow \OB$ be the isomorphism induced from $h$ 
defined in \cite{MaPacific}.
It satisfies $\Phi_h(\DA) = \DB$ and $\Phi_h = \Phi$ on $\DA$.
We then have by \eqref{eq:Psih} (\cite{MaPre2015})
\begin{equation*}
\Phi_h \circ \rho^{A,g}_t = \rho^{B,\Psi_{h^{-1}}(g)}_t \circ \Phi_h \quad 
\text{ for } g \in C(X_A,\Z).
\end{equation*}
 Since $\Psi_{h^{-1}}(1) = l_2 - k_2 =c_2 \in C(X_B,\Z)$
is  the cocycle function for $h^{-1}$,
we have
\begin{equation}
\Phi_h \circ \rho^{A}_t = \rho^{B,c_2}_t \circ \Phi_h, \qquad t \in \T. \label{eq:rhobc2}
\end{equation}
Put
$
\alpha = \Psi_h^{-1} \circ \Phi
$ 
which is an automorphism on $\OA$ such that
$\alpha |_{\DA} = \id$.
Let us next denote by 
$S_1,\dots,S_N$ the generating partial isometries 
of $\OA$ satisfying the relations \eqref{eq:CK}.
By \cite[Theorem 6.5 (1)]{MaPacific}, 
there exists a  unitary one-cocycle $V_\alpha(k) $ in $\DA$ such that 
$\alpha(S_\mu) = V_\alpha(k)S_\mu$ for $\mu\in  B_k(X_A).$
Hence we have
$\alpha(S_i) = V_\alpha(1) S_i, i=1,\dots,N$ so that 
\begin{equation}
\Phi(S_i) = \Phi_h(V_\alpha(1)S_i), \qquad i=1,\dots,N. \label{eq:PhiPhih} 
\end{equation}
By \eqref{eq:rhobct} and \eqref{eq:rhobc2}, 
the following equalities hold respectively
\begin{align}
e^{2\pi\sqrt{-1} t} \Phi(S_i)
&=\rho^{B,c_B}_t( \Phi(S_i))
= \Phi_h(V_\alpha(1)) \rho^{B,c_B}_t(\Phi_h(S_i)) \label{eq:rhobct1p} \\
\intertext{ and }
e^{2\pi\sqrt{-1} t}\Phi_h(V_\alpha(1)S_i)
= &\Phi_h(V_\alpha(1)\rho^{A}_t(S_i))  
= \rho^{B,c_2}_t(\Phi_h(V_\alpha(1)) \Phi_h(S_i)) \\
=&\Phi_h(V_\alpha(1)) \rho^{B,c_2}_t(\Phi_h(S_i)).
\label{eq:rhobc3}
\end{align}
By \eqref{eq:PhiPhih}, \eqref{eq:rhobct1p} and \eqref{eq:rhobc3},
we obtain
\begin{equation*}
\rho^{B,c_B}_t(\Phi_h(S_i)) =  \rho^{B,c_2}_t(\Phi_h(S_i)), \qquad i=1,\dots,N,\,  t \in \T. 
\end{equation*}
Thus we have
$c_2 = c_B$ a constant. 
By Lemma \ref{lem:useful},
we obtain that 
both cocycle functions $c_1$ and $c_2$ are constant $1$ so that 
$$
\Phi \circ \rho^A_t = \rho^B_t \circ \Phi, \qquad t \in \T.
$$


{ (ii) $\Longrightarrow$ (vi):\,\,} 
Let $h: X_A \rightarrow X_B$ 
be a homeomorphism 
giving rise to a continuous orbit equivalence between
$(X_A, \sigma_A)$ and $(X_B,\sigma_B)$.
By \cite{MaPacific},
one knows that 
$h \circ \Gamma_A \circ h^{-1} =\Gamma_B.
$
Since $h$ satisfies \eqref{eq:uorbiteq1K}
and \eqref{eq:uorbiteq2K},
for any $\tau_1 \in \Gamma_A$, there exists $K_{\tau_1} \in \N$ such that 
\begin{equation*}
\sigma_B^{K_{\tau_1}}(h \circ \tau_1 \circ h^{-1}(y)) =  
\sigma_B^{K_{\tau_1}}(y), \qquad y \in X_B.
\end{equation*}
Hence we have $h \circ \tau_1 \circ h^{-1} \in \Gamma_B^{\AF}.$
We similarly know that 
$h^{-1} \circ \tau_2 \circ h \in \Gamma_A^{\AF}$
for
$\tau_2 \in \Gamma_B^{\AF}$.


{ (vi) $\Longrightarrow$ (v):\,\,} This implication is obvious. 


{ (v) $\Longrightarrow$ (ii):\,\,} 
Suppose that 
there exists an isomorphism $\xi: \Gamma_A \longrightarrow \Gamma_B$
of groups such that $\xi(\Gamma_A^{\AF}) = \Gamma_B^{\AF}$.
By the proof of \cite[Theorem 7.2]{MaIsrael2015},
there exists a homeomorphism
$h:X_A\longrightarrow X_B$ which gives rise to $\xi$ such as 
\begin{equation}
\xi(\gamma)(y) = h(\gamma(h^{-1}(y))) \quad \text{ for } \gamma \in \Gamma_A, \, y \in X_B.
\end{equation}
Hence the actions
$\Gamma_A$ 
on
$X_A$
 and
$\Gamma_B$
on
$X_B$ 
are topologically conjugate
so that
$h \circ \Gamma_A\circ h^{-1} = \Gamma_B$.
By \cite{MaPacific},
the homeomorphism 
$h:X_A\longrightarrow X_B$ 
gives rise to a continuous orbit equivalence between 
$(X_A, \sigma_A)$ and $(X_B,\sigma_B)$.
By hypothesis, we have
$$
h \circ \tau_1 \circ h^{-1} = \xi(\tau_1) \in \Gamma_B^{\AF}\quad 
\text{ for } \tau_1 \in \Gamma_A^{\AF}.
$$
Hence 
for $\tau_1 \in \Gamma_A^{\AF}$,
there exists $K_{\tau_1}' \in \N$ such that 
\begin{equation}
\sigma_B^{K_{\tau_1}'}(h \circ \tau_1 \circ h^{-1}(y)) =  
\sigma_B^{K_{\tau_1}'}(y), \qquad y \in X_B.
\end{equation}
Put $x = h^{-1}(y)\in X_A$,
we have
\begin{equation}
\sigma_B^{K_{\tau_1}'}(h(\tau_1(x))) =  
\sigma_B^{K_{\tau_1}'}(h(x)), \qquad x \in X_A.
\end{equation}
Similarly we have
for $\tau_2 \in \Gamma_B^{\AF}$,
there exists $K_{\tau_2}' \in \N$ such that 
\begin{equation}
\sigma_A^{K_{\tau_2}'}(h^{-1}(\tau_2(y))) =  
\sigma_A^{K_{\tau_2}'}(h^{-1}(y)), \qquad y \in X_B.
\end{equation}
Hence 
the homeomorphism 
$h:X_A\longrightarrow X_B$ 
gives rise to a uniformly continuous orbit equivalence between 
$(X_A, \sigma_A)$ and $(X_B,\sigma_B)$.


Thefore we complete the proof of Theorem \ref{thm:main}.
 \qed


\section{Subclasses in continuous orbit equivalence class}

In this final section, we summarize relationships among several subequivalence relations  
in continuous orbit equivalence of one-sided topological Markov shifts. 
Suppose that $(X_A, \sigma_A)\underset{\coe}{\sim}(X_B,\sigma_B).$
Let $h:X_A\rightarrow X_B$ be a homeomorphism
and
$k_1,l_1: X_A\rightarrow \Zp,
 k_2,l_2: X_B\rightarrow \Zp 
$
be continuous functions satisfying \eqref{eq:orbiteq1x} and \eqref{eq:orbiteq2y}
respectively.
As in \cite{MMETDS},
the homomorphism
$\Psi_h: C(X_B,\Z) \rightarrow C(X_A,\Z)$
defined in \eqref{eq:Psih}
and its inverse 
$\Psi_{h^{-1}}: C(X_A,\Z) \rightarrow C(X_B,\Z)$
give rise to isomorphisms of the ordered abelian groups
between $H^A =C(X_A,\Z)/\{ g - g\circ\sigma_A \mid g \in C(X_A,\Z\} $ and $H^B$.
By definition of $\Psi_h$, we have that
$\Psi_h(1_{X_B})  = l_1 - k_1= c_1$ and similarly 
$\Psi_{h^{-1}} (1_{X_A})  = c_2$.
If $[c_1] = [1_{X_A}] $ in $H^A$
and
$[c_2] = [1_{X_B}] $ in $H^B$,
$(X_A, \sigma_A)$ and $(X_B,\sigma_B)$
 are said to be {\it strongly continuously orbit equivalent}\/, written
 $(X_A, \sigma_A)\underset{\scoe}{\sim}(X_B,\sigma_B)$
(\cite{MaJOT2015}).
If in particular $c_1 = 1_{X_A}$
and
$c_2 = 1_{X_B}$,
$(X_A, \sigma_A)$ and $(X_B,\sigma_B)$
 are eventually one-sided conjugate.
Then the following implications hold (cf. \cite{MaJOT2015}, \cite{MaPre2015}).
\medskip

\begin{align*}
&\hspace{58mm}\text{ UOE} \\
& \hspace{6cm} \Uparrow(0) \\
&\hspace{55mm}\text{ UCOE}\hspace{1cm}
  \overset{(1)}{\Longrightarrow} \hspace{1cm} \text{SCOE}\hspace{5mm}
   \overset{(2)}{\Longrightarrow} \hspace{3mm} \text{COE} \\
& \hspace{6cm} \Updownarrow(3) 
   \hspace{3cm} \Downarrow(4)  \\
& 
\text{one-sided conjugate} 
   \overset{(5)}{\Longrightarrow} \text{eventually one-sided conjugate}
    \hspace{5mm} \text{two-sided conjugate} 
\end{align*}

\medskip

The implications (0),  (1) , (2) and (5) are obvious.
The implications (3) come from Theorem \ref{thm:main}.
The implication (4) has been shown in \cite{MaJOT2015}.
Consider the following matrices
\begin{gather*}
A_2=
\begin{bmatrix}
1 & 1 \\
1 & 1
\end{bmatrix},
\qquad
F_2=
\begin{bmatrix}
1 & 1 \\
1 & 0
\end{bmatrix}, 
\qquad
B_2=
\begin{bmatrix}
1 & 1& 0 \\
1 & 0& 1 \\
1 & 0& 1 
\end{bmatrix}, \\
B_3=
\begin{bmatrix}
1 & 1& 1 \\
1 & 1& 1 \\
1 & 0& 0 
\end{bmatrix},
\qquad
C_3=
\begin{bmatrix}
1 & 1& 1 \\
1 & 1& 0 \\
1 & 1& 0 
\end{bmatrix},
\qquad
A_4=
\begin{bmatrix}
1 &1 & 1& 1 \\
1 &1 & 1& 1 \\
1 &1 & 1& 1 \\
1 &1 & 1& 1 
\end{bmatrix}.
\end{gather*}

Since
${\mathcal{F}}_{A_2}$ and ${\mathcal{F}}_{A_4}$
are the UHF algebras 
$M_{2^{\infty}}$ and $M_{4^{\infty}}$,
respectively,
there exists an isomorphism 
$\Phi:{\mathcal{F}}_{A_2} \longrightarrow {\mathcal{F}}_{A_4}$
such that
$\Phi({\mathcal{D}}_{A_2})={\mathcal{D}}_{A_4},$
so that 
$(X_{A_2},\sigma_{A_2}) \underset{\uoe}{\sim}(X_{A_4},\sigma_{A_4}).$
We however know that
${\mathcal{O}}_{A_2}= {\mathcal{O}}_{2} \ncong {\mathcal{O}}_{4}= {\mathcal{O}}_{A_4},$
so that 
$(X_{A_2},\sigma_{A_2}) \underset{\ucoe}{\not\sim}(X_{A_4},\sigma_{A_4}).$
Hence the converse of (0) does not necessarily hold.

By \cite{MaJOT2015}, we have shown that 
$
(X_{A_2},\sigma_{A_2}) \underset{\scoe}{\sim}(X_{B_2},\sigma_{B_2}).
$
As we know that
$$
K_0({\mathcal{F}}_{B_2}) = 
\Z^3 \overset{B_2^t}{\longrightarrow} \Z^3 \overset{B_2^t}{\longrightarrow}\cdots,
$$
 there exists an order preserving isomorphism
$\xi: K_0({\mathcal{F}}_{B_2}) \longrightarrow \Z[\frac{1}{2}] (\subset \R)$
such that 
$\xi([1]) =3 \in \R$.
Hence
$(K_0({\mathcal{F}}_{B_2}),[1]) \not\cong(K_0({\mathcal{F}}_{A_2}),[1])$
so that 
the AF algebra ${\mathcal{F}}_{B_2}$ is not isomorphic to 
${\mathcal{F}}_{A_2}$. 
This shows that
$
(X_{A_2},\sigma_{A_2}) \underset{\ucoe}{\not\sim}(X_{B_2},\sigma_{B_2}),
$
and the converse of (1) does not necessarily hold.

Since
${\mathcal{O}}_{A_2} \cong {\mathcal{O}}_{F_2} $ and $\det(\id -A_2) = \det(\id -F_2)$,
we have by Theorem \ref{thm:mm}
$(X_{A_2},\sigma_{A_2}) \underset{\coe}{\sim}(X_{F_2},\sigma_{F_2})$,
whereas
their two-sided topological Markov shifts
$(\bar{X}_{A_2},\bar{\sigma}_{A_2})$ and $(\bar{X}_{F_2},\bar{\sigma}_{F_2})$
are not topologically conjugate
so that
$(X_{A_2},\sigma_{A_2}) \underset{\scoe}{\not\sim}(X_{F_2},\sigma_{F_2}).$
Hence the converse of (2) does not necessarily hold.

Although the two-sided topological Markov shifts
$(\bar{X}_{B_3},\bar{\sigma}_{B_3})$ and $(\bar{X}_{C_3},\bar{\sigma}_{C_3})$
are topologically conjugate, 
we know that
${\mathcal{O}}_{B_3} \cong {\mathcal{O}}_3 
\ncong
{\mathcal{O}}_3\otimes M_2({\mathbb{C}})
\cong {\mathcal{O}}_{C_3}$
by \cite{EFW}. 
Hence the converse of (4) does not necessarily hold.
 
The converse of the implication (5) is an open question as in \cite{MaPre2015}.
\medskip

{\it Acknowledgments:}
This work was supported by JSPS KAKENHI Grant Number 15K04896.



\begin{thebibliography}{99}






















\bibitem{BH}
{\sc M. Boyle and D. Handelman},
{\it Orbit equivalence, flow equivalence and ordered cohomology},
Israel J.\  Math.\
{\bf 95}(1996), pp.\ 169--210.
































\bibitem{CK}{\sc J. ~Cuntz and W. ~Krieger},
{\it A class of $C^*$-algebras and topological Markov chains},
 Invent.\ Math.\
 {\bf 56}(1980), pp.\ 251--268.





%
\bibitem{EFW} {\sc M. Enomoto, M. Fujii and Y. Watatani},
{\it $K_0$-groups and classifications of Cuntz--Krieger algebras}, 
Math.\ Japon. {\bf 26}(1981), pp.\ 443--460.

















\bibitem{Kitchens}{\sc B.~P. ~Kitchens},
{\it Symbolic dynamics}, 
Springer-Verlag, Berlin, Heidelberg and New York (1998).


\bibitem{Kr}{\sc W. Krieger},
{\it On dimension functions and topological Markov chains},
Invent.\ Math. 
{\bf 56}(1980), pp.\ 239--250.


\bibitem{Kr2}{\sc W. Krieger},
{\it On dimension for a class of homeomorphism groups},
Math.\ Ann. {\bf 252}(1980), pp.\ 87--95.




\bibitem{LM}{\sc D. ~Lind and B. ~Marcus},
{\it An introduction to symbolic dynamics and coding},
 Cambridge University Press, Cambridge
(1995).












\bibitem{MaPacific}
{\sc K. Matsumoto},
{\it Orbit equivalence of topological Markov shifts and Cuntz--Krieger algebras},
Pacific J.\ Math.\ 
{\bf 246}(2010), 199--225.










\bibitem{MaDCDS}
{\sc K. Matsumoto},
{\it K-groups of the full group actions  
on one-sided topological Markov shifts},
 Discrete and Contin. Dyn. Syst.
{\bf 33}(2013),  pp.\ 3753--3765.


\bibitem{MaPAMS2013}
{\sc K. Matsumoto},
{\it  Classification of Cuntz--Krieger algebras by orbit equivalence of topological Markov shifts},
Proc. Amer. Math. Soc.
{\bf 141}(2013), pp.\ 2329--2342.


\bibitem{MaIsrael2015}{\sc K. Matsumoto},
{\it Full groups of one-sided topological Markov shifts},
Israel J. Math..
{\bf 205}(2015), pp. \ 1--33.

\bibitem{MaJOT2015}
{\sc K. Matsumoto},
{\it Strongly continuous orbit equivalence of 
one-sided topological Markov shifts},
J. Operator Theory {\bf 74}(2015), pp. 101--127.


\bibitem{MaPre2015}
{\sc K. Matsumoto},
{\it Continuous orbit equivalence, flow equivalence of Markov shifts and torus actions on Cuntz--Krieger algebras},
preprint, arXiv:1501.06965.  



\bibitem{MMKyoto}
{\sc K. Matsumoto and H. Matui},
{\it Continuous orbit equivalence of topological Markov shifts 
and Cuntz--Krieger algebras},
Kyoto J. Math.
{\bf 54}(2014), pp.\ 863--878.


\bibitem{MMETDS}
{\sc K. Matsumoto and H. Matui},
{\it Continuous orbit equivalence of topological Markov shifts 
and dynamical zeta functions}, preprint, arXiv:1403.0719,
to appear in Ergodic Theory Dynam. Systems.




\bibitem{MatuiPLMS}
{\sc H. Matui}, 
{\it Homology and topological full groups of {\'e}tale groupoids on totally disconnected spaces},
 Proc. London Math. Soc. {\bf 104}(2012), 
 pp.\ 27--56.

\bibitem{MatuiCrelle}
{\sc H. Matui}, 
{\it Topological full groups of one-sided shifts of finite type},
J. Reine Angew. Math. {\bf 705}(2015), pp. \ 35--84.













\bibitem{Po}
{\sc Y. T. Poon},
{\it A K-theoretic invariant for dynamical systems}, 
Trans.\  Amer.\ Math.\ Soc.\ 
{\bf 311}(1989), pp.\ 513--533.
















\bibitem{SV}{\sc S. Str{\v{a}}til{\v{a}} and D. Voiculescu},
{\it Representation of AF-algebras and of the groyup $U(\infty)$},
Lecture Notes in Mathematics, {\bf 486}(1975), Springer-Verlag, Berlin-New York.


\bibitem{Tomforde}{\sc M. Tomforde}
{\it The Graph Algebra Problem Page : List of Open Problems},
http://www.math.uh.edu/tomforde/GraphAlgebraProblems/ListOfProblems.html.









\end{thebibliography}
\end{document}